\documentclass[12pt,a4paper]{amsart}
\usepackage{amscd}
\usepackage{amssymb}
\usepackage[centering,text={15.5cm,22cm}]{geometry}
\usepackage{graphicx,color}
\usepackage[all]{xy}
\usepackage{mathrsfs}
\usepackage{marvosym}
\usepackage{stmaryrd}
\usepackage{srcltx}
\usepackage{hyperref}
\usepackage{verbatim}
\usepackage{graphicx}

\definecolor{shadecolor}{rgb}{1,0.9,0.7}

\newtheorem{theorem}{Theorem}[section]
\newtheorem{lemma}[theorem]{Lemma}
\newtheorem{lemma-definition}[theorem]{Lemma-Definition}
\newtheorem{proposition}[theorem]{Proposition}
\newtheorem{corollary}[theorem]{Corollary}

\theoremstyle{definition}
\newtheorem{setup}[theorem]{Setup}
\newtheorem{definition}[theorem]{Definition}

\theoremstyle{remark}
\newtheorem{remark}[theorem]{Remark}

\numberwithin{equation}{section}
\numberwithin{figure}{section}

\newcommand{\NN} {\mathbb{N}}
\newcommand{\ZZ} {\mathbb{Z}}

\newcommand{\CC} {\mathbb{C}}

\newcommand {\shH}  {\mathcal{H}}

\newcommand {\shM}  {\mathcal{M}}

\newcommand {\shO}  {\mathcal{O}}

\newcommand {\shS}  {\mathcal{S}}
\newcommand {\shT}  {\mathcal{T}}

\newcommand {\Ann}  {\operatorname{Ann}}

\newcommand {\Aut}  {\operatorname{Aut}}

\newcommand {\coker} {\operatorname{coker}}

\newcommand {\eps}  {\varepsilon}

\newcommand {\Hom}  {\operatorname{Hom}}

\renewcommand {\ker } {\operatorname{ker}}
\newcommand {\kk} {\Bbbk}

\newcommand {\liminv} {\varprojlim}

\newcommand {\llog} {\mathrm{log}}

\newcommand {\lra}  {\longrightarrow}

\newcommand {\op} {\operatorname}

\newcommand {\ra}  {\to}

\newcommand {\Spec} {\operatorname{Spec}}

\newcommand {\sra} {\twoheadrightarrow}

\def\mydate{\ifcase\month \or January\or February\or March\or
April\or May\or June\or July\or August\or September\or October\or 
November\or December\fi \space\number\day,\space\number\year}

\usepackage{scalerel,stackengine}
\stackMath
\newcommand\reallywidehat[1]{%
\savestack{\tmpbox}{\stretchto{%
  \scaleto{%
    \scalerel*[\widthof{\ensuremath{#1}}]{\kern-.6pt\bigwedge\kern-.6pt}%
    {\rule[-\textheight/2]{1ex}{\textheight}}
  }{\textheight}%
}{0.5ex}}%
\stackon[1pt]{#1}{\tmpbox}%
}

\begin{document}


\title
[Uniqueness of approximations, finite determinacy of log morphisms]
{Local uniqueness of approximations and finite determinacy of log morphisms}
\author{Helge Ruddat}

\address{JGU Mainz, Institut f\"ur Mathematik, Staudingerweg 9, 55099 Mainz, Germany}
\thanks{This work was financially supported by DFG Emmy-Noether grant RU 1629/4-1. The author thanks the IAS in Princeton, JGU Mainz and Univ. Hamburg for their hospitality.}
\email{ruddat@uni-mainz.de}

\maketitle

\newcommand{\bbb}[1]{\llbracket #1 \rrbracket}
\newcommand{\bbbt}{\bbb{t}}
\newcommand{\bbbQ}{\bbb{Q}}
\newcommand{\Def}{\operatorname{Def}}
\newcommand{\lcm}{\operatorname{lcm}}
\newcommand{\LT}{\operatorname{LT}}

\section{Introduction}
This article concerns the finite determinacy of the local structure of degenerating families in a sense that is best explained by the following two examples, for precise statements see \S\ref{section-main-results}.
A normal crossing degeneration is a flat family over a disc with coordinate $t$ that is locally of the form
$$z_0\cdot...\cdot z_i = t^N.$$
This family is \emph{determined at finite order} in the sense that any other flat family over the disc that is isomorphic to this one modulo $t^{k}$ for some $k>N$ is in fact entirely isomorphic locally near the origin. 
Indeed, by flatness, the new family is also given by a single equation and this is of the form
$$z_0\cdot...\cdot z_i = t^N+t^{k} f(t,z_0,...,z_n)$$
for $n$ the dimension of the fibres and $f(t,z_0,...,z_n)$ some perturbation term.
We can rewrite this as
$$z_0\cdot...\cdot z_i = t^N\underbrace{(1+t^{k-N} f(t,z_0,...,z_n))}_{=:g}$$
and since $g$ is invertible at the origin, we may absorb $g$ into one of the variables by a coordinate change of the form $z_0\mapsto g z_0$ to find this new family to be isomorphic to the original one locally at the origin. The purpose of this note is to generalize this result to situations where there are more equations than variables, making a similar ad hoc calculation as above very difficult.
The following feature observable in the above example will hold more generally as well: if the initial finite order isomorphism of the two families is modulo $t^{k}$ then the final all-order isomorphism of the two families agrees with the initial one modulo $t^{k-N}$. In the theorems below, this is actually $t^{k-4N}$ in part because we generally also treat non-complete intersections.

A slight modification of the above example makes the finite determinacy fail:
$$xy = w^Mt^N,$$
for some $M\ge 2$. Indeed, adding the perturbation term $t^{k}$ for some $k>N$ results in having smooth nearby fibres while the original family has an $A_{M-1}$-singularity as nearby fibre.
Thus, in proving the finite determinacy, we make the assumption that the nearby fibres are locally rigid (e.g. smooth). 
We however also provide a result that covers the second example in the following way. If the perturbation term is required to be divisible by $w^l$ for $l\ge M$, i.e.~geometrically, if the second family preserves a principal divisor containing the singularity, then finite determinacy holds. Indeed, similar to before,
$$xy = w^Mt^N + t^{k}w^l f(t,x,y,w)= w^Mt^N(1 + t^{k-N}w^{l-M} f(t,x,y,w)).$$

As an application of the precise approximation results in the next section, we prove in the final section (Theorem~\ref{main-log-result}) a more general version of the following result.
Let $B$ be regular and one-dimensional scheme over $\CC$ with distinguished point $0\in B$ and uniformizer $t$.
Let $f:X\ra B$ and $g:X'\ra B$ be flat finite type morphisms with $X$ a normal scheme and $f$ smooth away from $X_0=f^{-1}(0)$.
For $k\ge 0$, we denote by $f_k:X_k\ra B_k$ and $g_k:X'_k\ra B_k$ the base change of $f$ and $g$ to $B_k:=\Spec\shO_{B,0}/t^k$ respectively.

\begin{theorem} \label{theorem-intro} 
There is an integer $k>0$ depending only on $f$ so that if there is an isomorphism $\varphi:X'_k\ra X_k$ satisfying $g_k=f_k\circ \varphi$ then 
\begin{enumerate}
\item the maps of pairs $(X,X_0)\ra (B,0)$ and $(X',X'_0)\ra (B,0)$ are \'etale locally along $X_0$ and $X'_0$ isomorphic, so in particular 
\item $X'$ is also normal in a neighbourhood of $X'_0$ and the nearby fibres of $f'$ are also smooth and
\item for $X^\dagger_0$, ${X_0'}^\dagger$ and $0^\dagger$ the log spaces obtained from restricting the divisorial log structures of the pairs $(X,X_0), (X',X'_0), (B,0)$ to the divisor respectively, there is an isomorphism 
$\varphi^\dagger:{X_0'}^\dagger\ra X_0^\dagger$ over $0^\dagger$ whose underlying morphism is the reduction of $\varphi$ modulo $t$.
\end{enumerate}
\end{theorem}
This result is used in joint work with other authors \cite{RS19, smoothy} for studying analytic approximations of formal families. 
Theorem~\ref{theorem-intro} generalizes a theorem by Kisin \cite[Cor.~2.5]{kisin} who imposed the extra condition that $f$ be log smooth. Kisin's result in turn had been a generalization of a result by Illusie \cite[A.4]{nakayama}, \cite{illusie} who required the central fibre to be semistable. 
In the context of these prior works, we arrive at an interesting corollary.
The \emph{Kato-Kakayama space} is a topological space that is functorially associated to a log scheme over $\CC$ \cite{katonakayama,rounding}. 
In the situation of Theorem~\ref{theorem-intro}, assume the base ring is $\CC$ and we work with the complex topology. Restricting the log structures to central fibre $X_0:=f^{-1}(0)$ and the origin $0$ in $B$ respectively gives a log morphism $X_0^\dagger\ra 0^\dagger$ which is determined by the finite order $k$ in the theorem. Thus, the associated map of Kato-Nakayama spaces  $(f_0)_\llog:(X_0)_\llog\ra 0_\llog\cong S^1$ is also determined by the order $k$. 
On the other hand, for every $x\in 0_\llog$, the fibre $X_x:=(f_0)^{-1}_\llog(x)$ is expected to be homeomorphic to the nearby fibre of $f_0$, i.e.~a general fibre of $f$ and furthermore the topological family $(f_0)_\llog$ is expected to be the topological realization of the monodromy operation on the nearby fibre. 
To be more precise, let $r:X_x\ra (X_0)_\llog \ra X_0$ be the natural composition of embedding and projection. 
From the adjunction $\op{sp}:\ZZ \ra Rr_*\ZZ$ we obtain the exact triangle $\ZZ \ra Rr_*\ZZ \ra \op{cone}(\op{sp})\stackrel{[1]}{\ra}$ that is expected to be naturally isomorphic to the exact triangle
\begin{equation} \label{eq-nearby-vanishing}
\ZZ\ra\Psi_f\ZZ \ra \Phi_f\ZZ \stackrel{[1]}{\lra}
\end{equation}
of Deligne's nearby and vanishing cycles sheaves.
However, this hasn't been proved yet in the generality of Theorem~\ref{theorem-intro}. 
The most general situation where $(f_0)_\llog$ was shown to be a topological fibre bundle realizing the topological monodromy is when $f$ is \emph{relatively smooth}, a generalization of log smoothness given in \cite[Def.~3.6\,(2)]{rounding}.
The result is \cite[Theorem 3.7]{rounding}. We thus obtain the following corollary.
\begin{corollary} \label{corollary-intro}
Assume we are in the situation of Theorem~\ref{theorem-intro} over $\CC$ with $f$ relatively smooth then the nearby and vanishing cycle sequence $\eqref{eq-nearby-vanishing}$ is determined by the finite order $k$ in the sense that any other relatively smooth map $g$ with and isomorphism of the reductions of $f,g$ modulo $t^k$ induces a canonically isomorphic nearby and vanishing cycle sequence \eqref{eq-nearby-vanishing}.
\end{corollary}
The order $k$ in the above results is given explicitly in the next section.
The nearby cycle functor has been defined before in positive characteristic for a log smooth morphism \cite{nakayama}.
We hope that Theorem~\ref{theorem-intro} will provide the basis for future generalizations of the finite determinacy of the nearby cycle functor and the Lefschetz trace formula.

I would like to thank Klaus Altmann and Pierre Deligne for encouraging conversations. I am indebted to Duco van Straten for pointing out how Artin's theorem can be applied and for finding a serious error in an earlier version. 
I thank Bernd Siebert and Simon Felten for their careful reading and various suggestions for improvement. I thank Hubert Flenner for confirming that the main result was an open problem. I am grateful also to the anonymous quick opinion donator at Compositio for bringing to my attention Kisin's prior work and for suggesting various improvements of the presentation.


\section{Local uniqueness results}
\label{section-main-results}
We give three versions of the main result: a formal, an \'etale and a complex analytic one. Proofs for the following four results are provided in \S\ref{section-proofs}. In \S\ref{section-divisor}, we prove similar theorems where we allow the nearby fibre to be non-rigid but fix a principal divisor containing the non-rigid locus.
As an application, we prove the finite determinacy of the log structures on fibres of log morphisms in \S\ref{section-log-application}.

For $B$ a ring, $P=B[z_1,...,z_n]$, $A=P/J$ for an ideal $J\subset P$ and $M$ and $A$-module, one defines the $A$-module $T^1({A/B},M)=\coker(\op{Der}_B(P,M)\ra\Hom_A(J,M))$. The author's favorite reference for the $T^i$ functors is \cite{hartshorne}; the original source is \cite{lichtenbaumschlessinger}. We will give a review of the properties that we are going to use in the proof in \S\ref{section-proofs}. 
Let $\shT^1_{\Spec A/\Spec B}$ denote the coherent sheaf associated to $T^1({A/B},A)$.

\begin{setup} \label{the-setup}
Let $X\ra Y$ be a flat finite type map of Noetherian affine schemes and $t\in \Gamma(Y,\shO_Y)$ a non-zero-divisor. Flatness implies that the pull-back of $t$ is also a non-zero-divisor on $X$. Assume that $\shT^1_{X/Y}$ is supported in $t=0$ (i.e. it is annihilated by a power of $t$). 
We denote by $X_k\ra Y_k$ the base change of $X\ra Y$ to 
$\Spec \Gamma(Y,\shO_Y)/(t^k)$. We denote by $\widehat X\ra \widehat Y$ the completion of $X\ra Y$ in $t$.
Let $Z\subseteq X$ be a closed subscheme (possibly empty) so that $t$ is a non-zerodivisor also on $Z$.
\end{setup}

\begin{theorem}[\bf formal] 
\label{formal-result}
There exists $N>0$ so that if $X'\ra Y$ is another flat map of Noetherian affine schemes, $Z'\subseteq X'$ a closed subscheme and 
$\varphi:X'_k\cup Z'\ra X_k\cup Z$ an isomorphism over $Y$ for some $k>4N$ that restricts to an isomorphism $Z'\stackrel{\sim}\ra Z$, then there is an isomorphism $\widehat\varphi:\widehat X'\ra \widehat X$ that commutes with the maps to $\widehat Y$ and so that the restrictions of $\widehat\varphi$ and $\varphi$ to 
$\hat Z'\cup X'_{k-2N}$ agree where $\hat Z'$ is the completion of $Z'$ in $t$.
\end{theorem}

The authors of \cite{MvS02} prove a theorem similar to Theorem~\ref{formal-result} without the statement that the resulting isomorphism is compatible with the maps to $Y$, using rather different methods. 

For the next theorem, assume to work over a field or excellent discrete valuation ring.
\begin{theorem}[\bf \'etale local]
\label{etale-result}
There exists $N>0$ so that if $X'\ra Y$ is another flat map of Noetherian affine schemes, $Z'\subseteq X'$ a closed subscheme, $\varphi:X'_k\cup Z'\ra X_k\cup Z$ an isomorphism over $Y$ for some $k>4N$ that restricts to an isomorphism $Z'\stackrel{\sim}\ra Z$ and $x\in X_k$ a point, then there are \'etale neighbourhoods $U,U'$ of $x$ in $X,X'$ respectively and an isomorphism $\varphi_x:U'\ra U$ that commutes with the maps to $Y$ and so that $\varphi_x$ and $\varphi$ agree when restricting to $U'\times_{X'}(X'_{k-2N}\cup Z').$
\end{theorem}

\begin{theorem}[\bf analytic] 
\label{analytic-result}
Let $X\ra Y$ be a flat map of complex analytic spaces, $t\in \Gamma(Y,\shO_Y)$ a non-zero-divisor and $Z\subset X$ a closed complex analytic subvariety so that $t$ is a non-zerodivisor also on $Z$.
Assume that $\shT^1_{X/Y}$ is annihilated by a power of $t$. 
Let $Y_k$ denote the complex analytic space given by $t=0$ with sheaf of rings $\shO_Y/(t^k)$, similarly for $X_k$.

There exists $N>0$ so that if $X'\ra Y$ is another flat map of complex analytic spaces, $Z'\subset X'$ a closed analytic subvariety, $\varphi:X'_k\cup Z'\ra X_k\cup Z$ an isomorphism over $Y$ for some $k>4N$ that restricts to an isomorphism $Z'\stackrel{\sim}\ra Z$ and $x\in X_k$ a point, then there are neighbourhoods $U,U'$ of $x$ in $X,X'$ respectively and an isomorphism $\varphi_x:U'\ra U$ that commutes with the maps to $Y$ and so that 
$\varphi_x$ and $\varphi$ agree when restricting to $U'\cap(X'_{k-2N}\cup Z').$
\end{theorem}

We also prove the following useful criterion for the condition on the support of $\shT^1_{X/Y}$ as required in the above theorems.
\begin{lemma}[rigid nearby fibres]
\label{rigid-nearby-fibres}
Let $X\ra Y$ be a flat map of Noetherian schemes or of complex analytic spaces and $t$ a non-zero-divisor on $Y$.
If for all points $y\in U:=Y\setminus \{t=0\}$ we have $\shT^1_{X_y/y}=0$ for $X_y$ the fibre over $y$ then the support of $\shT^1_{X/Y}$ is contained in $t=0$.
\end{lemma}

\begin{corollary}[\bf uniqueness of neighbourhoods of points in singular fibres] 
Let $\pi:X\ra Y$ be a morphism of Noetherian schemes (resp. complex analytic varieties) and $x\in X$ so that $y:=\pi(x)$ is a regular point, $n:=\dim\shO_{Y,y}$ and assume in some neighbourhood $V$ of $y$ all fibres over $V\setminus\{y\}$ are locally rigid (i.e. have vanishing $\shT^1$).
Then there is $N>0$ so that if $\pi':X'\ra Y$ is another morphism and for some $k>4N$,
$$\varphi:X'\times_Y \Spec\shO_{Y,y}/\frak{m}_{Y,y}^{nk}\ra X\times_Y \Spec\shO_{Y,y}/\frak{m}_{Y,y}^{nk}$$ 
an isomorphism, then there is an isomorphism $\varphi_x:U'\ra U$ over $Y$ for $U',U$ \'etale neighbourhoods of $x$ in $X',X$ (respectively open analytic neighbourhoods) so that $\varphi_x$ agrees with $\varphi$ over $\Spec\shO_{Y,y}/\frak{m}_{Y,y}^{k-2N}$.
\end{corollary}
\begin{proof} 
By Lemma~\ref{rigid-nearby-fibres}, there is $N>0$ so that $\frak{m}_{Y,y}^N$ annihilates $\shT^1_{X/Y}$. Assume $k>4N$.
Let $\frak{m}_{Y,y}=(t_1,...,t_n)$, then $\frak{m}^{nk}_{Y,y}\subset (t_1^k,...,t_n^k)$ and $(t_1^{k-2N},...,t_n^{k-2N})\subset \frak{m}^{k-2N}_{Y,y}$. 
We only do the case $n=2$, the general case is similar. 
We first apply Theorem~\ref{etale-result} (resp. Theorem~\ref{analytic-result}) to the base change of $X\ra Y$ to the closed subspace $Y_1$ of $Y$ given by 
the ideal $(t_2^k)$ for $t=t_1$ and $Z=\emptyset$. 
The result is an isomorphism $\varphi_1$ of neighbourhoods of $x$ in the base changes of $X'\ra Y$ and $X\ra Y$ to $Y_1$ that agrees with $\varphi$ modulo $(t_1^{k-2N},t_2^k)$.
In the second step, we apply the corresponding theorem to $X'\ra Y$ and $X\ra Y$ with $t=t_2$ using $\varphi_1$ as input and obtaining the desired isomorphism as a result.
\end{proof}


\section{Proofs of the local uniqueness results} \label{section-proofs}
For $B\ra A$ a map of rings and $M$ an $A$-module, there is an $A$-module $T^i(A/B,M)$ for $i=0,1,2$ defined in \cite{lichtenbaumschlessinger}. More generally, $T^i$ is the $i$th cohomology module of $\Hom(L^\bullet,M)$ for $L^\bullet$ a cotangent complex of $B\ra A$.
One sets $T^i_{A/B}:=T^i({A/B},A)$.
We recall some standard properties of these $T^i$ functors for later use, proofs for these can be found in \cite{lichtenbaumschlessinger} and \cite{hartshorne}. 

\subsection{$\mathbf{T^0}$} \label{Tzero-section}
For $B\ra A$ a map of rings and $M$ an $A$-module: $T^0(A/B,M)=\op{Der}_B(A,M)$.

\subsection{Smoothness} \label{smooth-section}
For $B\ra A$ smooth, one has $T^i(A/B,M)=0$ for all $M$ and $i>0$.

\subsection{Finite generatedness} \label{finite-generation}
For $B$ Noetherian, $A$ a finitely generated $B$-algebra and $M$ a finitely generated $A$-module, $T^i(A/B,M)$ is finitely generated for $i=0,1,2$.

\subsection{Closed embeddings} \label{close-embedding}
If $B\ra A$ is surjective with ideal $I$, then for all $M$, 
$$
\begin{array}{rcl}
T^0(A/B,M)&=&0,\\
T^1(A/B,M)&=&\Hom_A(I/I^2,M),
\end{array}
$$
and if $I$ is generated by a regular sequence then $T^2(A/B,M)=0$ for all $M$.

\subsection{Change of rings} \label{change-of-rings}
If $A\ra B\ra C$ are maps of rings and $M$ is a $C$-module, there is a long exact sequence of $C$-modules
$$
\begin{array}{lllllllll}
0&\ra &T^0(C/B,M)&\ra& T^0(C/A,M)&\ra& T^0(B/A,M)\\
&\ra &T^1(C/B,M)&\ra& T^1(C/A,M)&\ra& T^1(B/A,M)\\
&\ra &T^2(C/B,M)&\ra& T^2(C/A,M)&\ra& T^2(B/A,M).\\
\end{array}
$$

\subsection{Change of modules} \label{change-of-modules}
If $B\ra A$ is a map of rings and $0\ra M\ra M'\ra M''\ra 0$ an exact sequence of $A$-modules, there is a long exact sequence
$$
\begin{array}{lllllllll}
0&\ra &T^0(A/B,M)&\ra& T^0(A/B,M')&\ra& T^0(A/B,M'')\\
&\ra &T^1(A/B,M)&\ra& T^1(A/B,M')&\ra& T^1(A/B,M'')\\
&\ra &T^2(A/B,M)&\ra& T^2(A/B,M')&\ra& T^2(A/B,M'').\\
\end{array}
$$

\subsection{Base change} \label{lemma-basechange}
Let $A$ be a flat $B$-algebra, $B'$ any $B$-algebra, $A'=A\otimes_B B'$ and $M'$ an $A'$-module then
$$T^1(A/B,M')=T^1(A'/B',M').$$

\subsection{Preparations for proving the statements in \S\ref{section-main-results}}
\label{subsec-prep-proof}
Assume now we are in a situation as in Setup~\ref{the-setup}. 
Set $B=\Gamma(Y,\shO_Y)$, $A=\Gamma(X,\shO_X)$, $A'=\Gamma(X',\shO_X')$.
Let $I\subset A$ be the ideal of $Z$ and $I'\subset A'$ be the ideal of $Z'$.
Since $t$ is assumed to be a non-zerodivisor in $A/I$, we have $I\cap (t^l)=t^lI$ for every $l$. Similarly, $I'\cap (t^l)=t^lI'$ for every $l$.
The given isomorphism $\varphi$ at ring level, we also call $\varphi:A/t^kI\ra A'/t^kI'$.

Since $T^1({A/B},I)$ is finitely generated by \S\ref{finite-generation}, 
by assumption there is $N>0$ such that $t^NT^1({A/B},I)=0$.
Let $\ker_{t^l}T^2({A/B},I)$ denote the kernel of the endomorphism of $T^2({A/B},I)$ given by multiplication by $t^l$.
Since also $T^2({A/B},I)$ is finitely generated by \S\ref{finite-generation}, the sequence 
\begin{equation}
\label{increasing-sequence-to-define-stable}
\ker_{t}T^2({A/B},I)\subseteq \ker_{t^2}T^2({A/B},I)\subseteq \ker_{t^3}T^2({A/B},I)\subseteq ...
\end{equation}
stabilizes and, by increasing $N$ if needed, we may assume $\ker_{t^N}T^2({A/B},I)$ is this stable $A$-module.

\begin{lemma}
\label{lemma-shift}
Let $L,l\ge 0$ be given, we have that 
$I/t^{L}I$ and $t^{l}I/t^{l+L}I$ are isomorphic $A$-modules.
\end{lemma}
\begin{proof} Since $t$ is a non-zerodivisor, multiplication by $t^l$ is injective in $I$ and thus yields an isomorphism $I\cong t^lI$. The first two vertical maps in the following diagram are isomorphisms, hence also the third one is
$$
\xymatrix{
0\ar[r] & t^LI \ar[r]\ar_{\cdot t^l}[d] & I \ar_{\cdot t^l}[d] \ar[r] & I/t^LI \ar@{-->}_{\cdot t^l}[d] \ar[r] & 0\\
0\ar[r] & t^{l+L}I \ar[r] & t^lI \ar[r] & t^lI/t^{l+L}I \ar[r] & 0.
}$$
\end{proof}

By \S\ref{change-of-modules}, the short exact sequence of $A$-modules
$0\ra I\stackrel{\cdot t^L}\lra I \ra I/t^LI \ra 0$
induces a long exact sequence 
\begin{equation} \label{eq-les-talpha}
 ...\ra T^1({A/B},I) \stackrel{\cdot t^L}{\lra} T^1({A/B},I)\ra T^1({A/B},I/t^LI) \ra T^2({A/B},I) \stackrel{\cdot t^L}{\lra} T^2({A/B},I)   
\end{equation}
and thus for $L>N$ a short exact sequence
$$0\ra T^1({A/B},I)\ra T^1({A/B},I/t^LI) \ra \ker_{t^N}T^2({A/B},I)\ra 0.$$
Since multiplication by $t^N$ acts on this sequence and is zero on the outer terms, we conclude the following result.
\begin{lemma} 
For $L\ge N$ we have $t^{2N}T^1({A/B},I/t^{L}I)=0$ and hence furthermore, for any $l\ge 2L$, by Lemma~\ref{lemma-shift},
$$
t^{2N}T^1({A/B},t^{l-L}I/t^{l}I)=0
$$
\end{lemma}
Set $A_L=A/t^LA$ and similarly for $A'_L,B_L$.
Note that $t^{l-L}I/t^{l}I\cong I/t^{L}I$ is an $A_L$-module. 
If we assume that $k\ge L$ then base change \S\ref{lemma-basechange} gives 
\begin{equation*}
\resizebox{\textwidth}{!}{$T^1({A/B},t^{l-L}I/t^{l}I)=T^1({A_L/B_L},t^{l-L}I/t^{l}I) \stackrel\varphi= T^1({A'_L/B_L},t^{l-L}I'/t^{l}I')=T^1({A'/B},t^{l-L}I'/t^{l}I'),$}
\end{equation*}
so we conclude the following lemma.
\begin{lemma} 
\label{annnili1}
For $k\ge L\ge N$ and $l\ge 2L$, we have
$$
t^{2N}T^1({A'/B},t^{l-L}I'/t^{l}I')=0
$$
\end{lemma}

Now let 
\begin{equation}
\label{eq-presentation}
A=B[z_1,\ldots,z_n]/(f_1,\ldots, f_m)
\end{equation}
be a presentation as a $B$-algebra. 
Let $\reallywidehat{B[z_1,...,z_n]}$, $\widehat A'$, $\widehat B$ denote the completions in $t$ of $B[z_1,...,z_n]$, $ A'$, $B$ respectively.
Let $\hat{z_i}\in A'$ be a lift of the image of $z_i$ under the composition $A\sra A/t^k I \stackrel{\varphi}{\lra} A'/t^kI'$.
\begin{lemma} 
\label{lemma-Acirc}
The map $\pi:\reallywidehat{B[z_1,...,z_n]}\ra \widehat A', z_i\mapsto \hat z_i$ is surjective.
\end{lemma}
\begin{proof} The proof hinges on the following property. 
For any $a'\in A'$, there is $b\in A'$ and a polynomial $F$ in the $\hat z_i$ with coefficients in $B$ so that
$$a'-t^kb=F.$$
This property follows from the fact that, by construction, the $\hat{z_i}$ generate $A'_k$ as a $B$-algebra.
In fact, by iterating, i.e. reapplying the property to $b$ and so forth, we have the more general property that
for any $a'\in A'$ and $s\ge 1$, there is $b\in A'$ and a polynomial $F$ in the $\hat z_i$ with coefficients in $B$ so that
\begin{equation}
\label{eq-shift-by-k}
a'-t^{sk}b=F.
\end{equation}

Let $J$ denote the kernel of $\pi$ and set $J_l:=\ker\big(B[z_1,...,z_n]/(t^l)\ra A'_l\big)$. 
We claim that the natural map $J_l\ra J_{l-1}$ is surjective for all $l$. 
Indeed, let $\bar a\in J_{l-1}$ be given and represented by $a\in B[z_1,...,z_n]$.
Hence, $\pi(a)=t^{l-1}a'$ for some $a'\in A'$. By \eqref{eq-shift-by-k}, 
$t^{l-1}a'=t^{l+k-1}B$ for some $B\in A'$, so $a$ descends to an element of $J_l$ and this maps to $\bar a$.

Since the completion functor is left exact, we have $J=\liminv_l J_l$. 
If we show the surjectivity of
$B[z_1,...,z_n]/(t^l)\ra A'_l$ for all $l\ge 0$, then we have for all $l$ an exact sequence
$$0\ra J_l\ra B[z_1,...,z_n]/(t^l)\ra A'_l\ra 0.$$
Since the inverse system $J_l$ has surjective projections as we just verified, taking $\liminv_l$ on of the exact sequence yields again an exact sequence which is in fact 
$0\ra J\ra \reallywidehat{B[z_1,...,z_n]}\ra \widehat A'\ra 0$
thus the assertion of the Lemma follows.
Hence, it suffices to prove that $B[z_1,...,z_n]/(t^l)\ra A'_l$ is surjective for all $l\ge 0$. 
Now let any $l\ge 0$ and an element $a'\in A'$ be given. 
Choose $s$ so that $sk\ge l$.
By \eqref{eq-shift-by-k}, we find
$b\in A'$ and $F$ a polynomial in the $\hat z_i$ so that
$a'-t^{sk}b=F$.
We conclude that the projection of $a'$ in $A_l$ is the image of $ F(z_1,...,z_n)\in B[z_1,...,z_n]$. This proves the desired surjectivity.
\end{proof}

\begin{lemma}  \label{lemma-flatness-completion}
$\widehat B\ra \widehat A'$ is flat.
\end{lemma}
\begin{proof} By assumption $B\ra A'$ is flat and thus the base change $\widehat B\ra A'\otimes_B \widehat B$ is also flat. 
Since all rings involved are Noetherian, $A'\otimes_B \widehat B \ra (A'\otimes_B \widehat B)^{\widehat{\,}}$ is also flat, e.g. by \cite[Proposition 10.14]{AM}. Since compositions of flat maps are flat, we are done if we show that 
$\widehat A'\cong (A'\otimes_B \widehat B)^{\widehat{\,}}$. This follows by identifying the inverse systems that give these inverse limits,
$$
(A'\otimes_B \widehat B) / t^l (A'\otimes_B \widehat B)
\cong (A'\otimes_B \widehat B)\otimes_{\widehat B} B/t^l B
\cong A'\otimes_B (B/t^l B)
\cong A'/t^l A'.
$$
\end{proof}

\begin{proposition}[Lifting equations]
\label{lemma-gi}
We find $g_1,...,g_m\in \reallywidehat{B[z_1,...,z_n]}$, so that
$$\widehat A'= \reallywidehat{B[z_1,...,z_n]}/(f_1+t^{k}g_1,\,\,...\,\,,f_m+t^kg_m).$$
\end{proposition}
\begin{proof} 
By Lemma~\ref{lemma-Acirc}, the map $\reallywidehat{B[z_1,...,z_n]}\ra \widehat A', z_i\mapsto \hat z_i$ is surjective, let $J$ be the kernel. 
By Lemma~\ref{lemma-flatness-completion},
$\widehat A'$ is flat over $\widehat B$, hence
$\operatorname{Tor}^{\widehat B}_1(B_{k},\widehat A')=0$.
Tensoring the exact sequence of $\widehat B$-modules
$$0\ra J \ra  \reallywidehat{B[z_1,...,z_n]}\ra \widehat A'\ra 0$$
with $B_k$ thus yields again an exact sequence
$$
0\ra J/t^kJ \ra  B[z_1,...,z_n]/(t^k)\ra A'_k\ra 0.
$$
From  $A'_k\stackrel{\varphi}{=}A_k$ we conclude 
\begin{equation}
\label{eq-eq-of-ideals}
J+(t^k)=(f_1,...,f_m)+(t^k)
\end{equation}
as an equality of ideals of $\reallywidehat{B[z_1,...,z_n]}$.
Hence there exist $g_i$ such that $f_i+t^k g_i\in J$. 
Set $J'=(f_1+t^k g_1,...,f_m+t^k g_m)$. We want to show that $J'=J$.
Define $R$ by the exact sequence
\begin{equation} \label{seq-def-R}
0\ra J'\ra J\ra R \ra 0
\end{equation}
of $\reallywidehat{B[z_1,...,z_n]}$-modules, 
and we want to show $R=0$. 
Tensoring \eqref{seq-def-R} with  $B_k$ yields
$$J'\otimes B_k\stackrel{\alpha}{\lra} J\otimes B_k\ra R\otimes B_k \ra 0.$$ 
Note that by \eqref{eq-eq-of-ideals}, the map $\alpha$ is surjective and thus $R\otimes B_{k}=0$. 
Now let $r\in R$ be any element. Since $r\otimes 1$ is zero in $R\otimes B_{k}$, we have
$r=t^k r'$ for some $r'\in R$. 
Similarly, $r'=t^kr''$ and so forth until we have a sequence $r=t^{sk}r^{(s)}$ for $s\ge 0$. This means 
$r\in\bigcap_{l\ge 0}t^lR$.
Note that $J'$ and $J$ are finitely generated $\reallywidehat{B[z_1,...,z_n]}$-modules, hence $R$ is finitely generated and hence $t$-adically complete (e.g. by \cite[Proposition 10.13]{AM}), i.e. $R=\liminv (R/t^l R)$.
However this implies $\bigcap_{l\ge 0}t^lR=0$, hence $r=0$.
\end{proof}

\begin{proposition}[Lifting relations] 
\label{prop-lift-relations}
Given $l>0$, let $a_1,...,a_m\in B[z_1,...,z_n]$ be such that
$$a_1f_1+...+a_mf_m \in (t^l).$$
Then there are $a_i'\in B[z_1,...,z_n]$ such that
$$(a_1+t^la_1')f_1+...+(a_m+t^la_m')f_m=0.$$
A similar statement holds if we replace $B[z_1,...,z_n]$ by $\reallywidehat{B[z_1,...,z_n]}$, and also if we then additionally replace $f_i$ by $f_i':=f_i+t^kg_i$ in view of Proposition~\ref{lemma-gi}.
\end{proposition}
\begin{proof}
Consider the exact sequence of $B$-modules
$0\ra J\ra B[z_1,...,z_n]\ra A\ra 0$
since all terms except possibly $J$ are flat, we deduce that also $J$ is a flat $B$-module.
Let $...\ra R_{-2}\ra R_{-1}\ra R_0\ra J\ra 0$ be a resolution of $J$ by finitely generated free $B[z_1,...,z_n]$-modules. We assume $R_0$ is freely generated by $f_1,...,f_m$.
All terms in the sequence are flat $B$-modules, hence tensoring the sequence with $B_l$ yields another exact sequence. We have $(a_1,...,a_m)\in \ker(R_{0}\otimes_B B_l\ra J\otimes_B B_l)$, 
so it comes from an element $b\in R_{-1}\otimes_B B_l$. Let $\hat b$ be a lift of $b$ in $R_{-1}$. Its image in 
$R_{0}$ takes the form $(a_1+t^la_1',...,a_m+t^la_m')$ and it maps to zero in $J$, so is a relation as desired.
\end{proof}

\subsection{Proofs for the statements in \S\ref{section-main-results}}
\label{proof-section}
We continue with the setup of the previous paragraph. 
The first draft of the following main result was inspired by \cite[Proof of Prop. 4.4]{hartshorne}.

\begin{proposition}[Lifting isomorphisms] 
\label{prop-lift-iso}
Assume $l\ge k$ and that for some $h$ we have a commutative diagram of $\widehat B$-algebras
\[
\vcenter{\vbox{
\xymatrix@C=30pt
{ 
 \reallywidehat{B[z_1,\ldots,z_n]}\ar^{z_i\mapsto \hat z_i}[r]\ar_h[d] & A'/t^{l+1}I' \ar[r]& A'/t^{l}I'\ar_\beta[d] \ar[r]& A'/t^{k}I' \ar_{\varphi^{-1}}[d]\\
 \reallywidehat{B[z_1,\ldots,z_n]}\ar^{z_i\mapsto z_i}[r]
 & A/t^{l+1}I \ar[r]& A/t^{l}I \ar[r]& A/t^{k}I 
}
}}
\]
where the horizontal maps are the natural projections (use Lemma~\ref{lemma-Acirc}) and $\beta$ is an isomorphism. Let $\bar\beta:A'/t^{l-2N}I'\ra A/t^{l-2N}I$ denote the induced isomorphism.
Then there is a $\widehat B$-linear isomorphism $\hat\beta:A'/t^{l+1}I' \ra A/t^{l+1}I$ making the following diagram commutative
\[
\vcenter{\vbox{
\xymatrix@C=30pt
{ 
A'/t^{l+1}I' \ar_{\hat\beta}[d]\ar[r]& A'/t^{l-2N}I' \ar_{\bar\beta}[d]\\
A/t^{l+1}I \ar[r]& A/t^{l-2N}I.\\
}
}}
\]
\end{proposition}
\begin{proof}
First note that only $\beta$ is the relevant datum since $h$ can always be found for any given $\beta$.
We find for all $l$ that 
$$A/t^{l}I= A/t^{l}\times_{A/((t^{l})+I)} A/I \qquad \hbox{ and } \qquad A'/t^{l}I'= A'/t^{l}\times_{A'/((t^{l})+I')} A'/I'.$$
Note that the isomorphism $\beta$ respects these fibre product decompositions, i.e. $\beta$ decomposes in isomorphisms $\beta_1:A'/t^{l}\ra A/t^{l}$, 
$\beta_2:A'/((t^{l})+I')\ra A/((t^{l})+I)$ and $\beta_3:A'/I'\ra A/I$.
In the pursuit of producing $\hat\beta$, requiring $\hat\beta_3=\beta_3$, 
it suffices to produce an isomorphism $\hat \beta_1:A'/t^{l+1}\ra A/t^{l+1}$ that restricts to the isomorphism
$\hat \beta_2:A'/((t^{l+1})+I')\ra A/((t^{l+1})+I)$ that is already induced from $\beta_3$.

Here is how we achieve this. From now on, we mostly work modulo $t^{l+1}$, first some notation: 
set 
$$\begin{array}{rcl}J'&:=&\ker\big(B[z_1,\ldots,z_n]/t^{l+1}\ra A'/t^{l+1}\big),\\
J&:=&\ker\big(B[z_1,\ldots,z_n]/t^{l+1}\ra A/t^{l+1}\big),
\end{array}
$$
so $J=(f_1,...,f_m)$. By Proposition~\ref{lemma-gi}, $J'=(f'_1,...,f_m')$ with $f_i'=f_i+t^kg_i$. 
By abuse of notation, we also denote by $I$ and $I'$ the pullback of the respective ideal to $B[z_1,\ldots,z_n]/t^{l+1}$ and also to $\reallywidehat{B[z_1,\ldots,z_n]}$.
Let $h_{l+1}:B[z_1,\ldots,z_n]/t^{l+1}\ra B[z_1,\ldots,z_n]/t^{l+1}$ be the restriction of $h$.
Since $\beta$ is an isomorphism, we have that $h_{l+1}$ as well as $h$ identifies $t^{l}I'+J'$ with $t^{l}I+J$ and also their squares. 
We thus have the following inclusions of $B[z_1,\ldots,z_n]/t^{l+1}$-modules,
\[
\vcenter{\vbox{
\xymatrix@C=30pt
{ 
 J'^2 \ar@{}[r]|-{\subset}
& (t^{l}I'+J')^2 \ar^\sim_h[d]\ar@{}[r]|-{\subset}   
 &   J' \ar@{}[r]|-{\subset}  
 &   t^{l}I'+J'  \ar^\sim_h[d]  \\
& (t^{l}I+J)^2 \ar@{}[r]|-{\subset}
  &    J \ar@{}[r]|-{\subset}
  &   t^{l}I+J \ar@{}[r]|-{\subset}
  &  t^{l-2N}I+J.
}
}}
\]
We observe that $h$ induces a $B[z_1,\ldots,z_n]/t^{l+1}$-module homomorphism 
$$\bar h\in \Hom\left(\frac{J'}{J'^2},\frac{t^{l-2N}I+J}{t^{l+1}I+J}\right)=:H.$$
\noindent\underline{Claim:} $\bar h=t^{2N} \bar h'$ for some $\bar h'\in H$.\\
Indeed, $h$ sends $f_i'$ to $h(f_i')=f_i+t^l h_i\equiv t^l h_i\ (\op{modulo} J)$ for some $h_i\in I$ (as an element of $\reallywidehat{B[z_1,\ldots,z_n]}$).
We define $\bar h'$ via $\bar h'(f_i')=t^{l-2N} h_i$ and need to show this is well-defined. 
Since $h_i\in I$, $\bar h'$ does indeed map into $t^{l-2N}I+J$ and we need to check that this preserves relations.
Let $0=\sum a_if'_i$ hold in $B[z_1,\ldots,z_n]/t^{l+1}$, 
so by Prop.~\ref{prop-lift-relations}, we may assume this holds already in $\reallywidehat{B[z_1,\ldots,z_n]}$.
Hence, $0=\sum h(a_i)h(f'_i)=\sum h(a_i)(f_i+t^l h_i)$ holds in $\reallywidehat{B[z_1,\ldots,z_n]}$ and thus
$0=t^{2N}\sum h(a_i)(t^{l-2N} h_i)$ holds in $A$ where $t^{2N}$ is a non-zerodivisor. We conclude that
$0=\sum h(a_i)(t^{l-2N} h_i)$ holds in $A$ and then also in $A/t^{l+1}$. Hence $\bar h'$ is well-defined and we verified the claim.

We derive from \S\ref{change-of-rings}, \S\ref{Tzero-section}, \S\ref{smooth-section} and \S\ref{close-embedding} the following exact sequence of $B[z_1,...,z_n]/t^{l+1}$-modules 
\begin{equation}
\label{eq-seq-T1}
\resizebox{.91\textwidth}{!}{$
\Hom\left(\Omega_{(B[z_1,\ldots,z_n]/t^{l+1})/B_{l+1}},\frac{t^{l-2N}I+J}{t^{l+1}I+J}\right)\ra H 
\ra T^1\left(A'_{l+1}/B_{l+1},\frac{t^{l-2N}I+J}{t^{l+1}I+J}\right)\ra 0.
$}
\end{equation}
Note that $\frac{t^{l-2N}I+J}{t^{l+1}I+J}$ is an $A_{2N+1}$-module that is identified with the $A'_{2N+1}$-module $\frac{t^{l-2N}I'+J'}{t^{l+1}I'+J'}$ under $\varphi$.
Setting $L=2N+1$ in Lemma~\ref{annnili1} combined with the claim above yields that $\bar h$ maps to zero in $T^1\left(A'_{l+1}/B_{l+1},\frac{t^{l-2N}I+J}{t^{l+1}I+J}\right)$ and hence lifts to a derivation
$$\theta\in\Hom_{B[z_1,...,z_n]/t^{l+1}}\big(\Omega_{(B[z_1,...,z_n]/t^{l+1})/B_{l+1}},(t^{l-2N}I+J)/(t^{l+1}I+J)\big).$$
Then let $\tilde\theta$ refer to the map  of $B$-modules
$B[z_1,...,z_n]/t^{l+1}\ra A_{l+1}$,  $q\mapsto \theta(dq).$
We then set $\hat\beta_1=h-\tilde\theta$. Indeed this gives a map $B[z_1,...,z_n]/t^{l+1}\ra A_{l+1}$ that sends $J'$ to zero since by construction the image of $J'$ under $h$ and under $\tilde\theta$ agree. 
We therefore get a map $\hat\beta_1 :A'_{l+1}\ra A_{l+1}$ of $B$-modules.
First observe that $\hat\beta_1$ is a map of $B$-algebras:
indeed, we find $(h(a)-\tilde\theta(a))(h(b)-\tilde\theta(b))=h(ab)-\tilde\theta(ab)$ follows from $\tilde\theta(a)\tilde\theta(b)=0$ since $2(l-2N)\ge l+1$.
Since $\beta$ restricts to an isomorphism $A'/I'\ra A/I$ and $h$ is a lift of it, the restriction of $h$ to $A'_{l+1}\ra A_{l+1}$ takes $I'$ to $I$. 
Since $\theta$ has image contained in $I$, we conclude that $\hat\beta_1$ induces the restriction $A'/((t^{l+1})+I')\ra A/((t^{l+1})+I)$ which coincides with the restriction of $\beta$. This readily allows us to glue $\hat\beta_1$ along $\hat\beta_2$ and $\hat\beta_3$ to a map $\hat\beta:A'/t^{l+1}I'\ra A/t^{l+1}I$.
Furthermore, because $\theta$ has image in $t^{l-2N}I$, $\hat\beta$ fits in the commutative diagram with $\bar\beta:A'/t^{l-2N}I'\ra A/t^{l-2N}I$ as required in the assertion. 

It remains to show that $\hat\beta$ is an isomorphism, or equivalently that $\hat\beta_1$ is one.
We know that modulo $t^k$, the map $\hat\beta_1$ agrees with the restriction of $\varphi$. 
By the definition of $\hat z_i$, we thus have $h(z_i)=z_i+t^kz_i'$ and thus $h(z_i)-\tilde\theta(z_i)=z_i+t^kz_i'+t^lz_i''$ for some $z_i',z_i''$. 
We can view $h(z_i)-\tilde\theta(z_i)$ as a lift of $z_i$ in the sense of Lemma~\ref{lemma-Acirc} which then implies that $\hat\beta_1$ is surjective. 
We still have to show that $\hat\beta_1$ is also injective. Let $\widehat\beta_1:B[z_1,...,z_n]/t^k\ra B[z_1,...,z_n]/t^k$ be a lift of $\hat\beta_1$ (which sends $J'$ to $J$).
Since $\beta$ is an isomorphism, an element in the kernel of $\hat\beta_1$ is can be represented in the form $t^lf$ for some $f\in B[z_1,...,z_n]/t^{l+1}$. 
By flatness (as argued for $I$ before), $(t^{l})\cap J=t^lJ$, so $\widehat\beta_1$ maps $t^lf$ into $t^lJ$. 
But on the other hand, since $\varphi$ is an isomorphism, modulo $t$, $\widehat\beta_1$ is identity and $J'/tJ'=J/tJ$. Thus
$f+tf'\in J'$ for some $f'$. But then $t^l(f+tf')=t^lf$ in $A'_{l+1}$, so $t^lf=0$ in $A'_{l+1}$.
\end{proof}

\begin{proof}[Proof of Theorem~\ref{formal-result}]
Using Proposition~\ref{prop-lift-iso}, an induction on $l\ge k$ with $l=k$ the base case gives an isomorphism $\varphi_l:A'_l\ra A_l$ for $l\ge k$.
These isomorphisms are not compatible with the projections $A'_{l+1}\ra A'_l$ and $A_{l+1}\ra A_l$, however this is easy to fix.
Define $\psi_l:=\varphi_{l+2N+1}/t^{l}: A'_l\ra  A_l$. Then for any $l\ge 0$ we have a commutative diagram
\[
\xymatrix@C=30pt
{ 
A'_{l+1}\ar_{\psi_{l+1}}[d]\ar[r] &  A_l\ar_{\psi_{l}}[d]\\
A'_{l+1}\ar[r] &  A_l
}
\]
where the vertical maps are isomorphisms and the horizontal maps the natural projections. Taking the inverse limit over $l$ gives the result.
\end{proof}

\begin{proof}[Proof of Theorem~\ref{etale-result}]
In view of \cite[Corollary (2.1)]{artin1}, Theorem~\ref{etale-result} is essentially a corollary of Theorem~\ref{formal-result} as follows.
Present $A=B[z_1,...,z_n]/(f_1,...,f_m)$, $A'=B[z_1,...,z_n]/(g_1,...,g_m)$ so that $x$ is the origin. 
Consider the system of $m$ equations 
\begin{equation}
\label{system-to-be-solved-by-graph}
f_i(a_1,...,a_n)= \sum_{j=1}^m b_{ij}g_j
\end{equation}
in the variables $a_i,b_{ij}$. 
The $\widehat B$-linear isomorphism $\hat\varphi$ of $\widehat A$ with $\widehat A'$ obtained by Theorem~\ref{formal-result} provides a formal solution of \eqref{system-to-be-solved-by-graph} by setting 
$a_j=\hat\varphi_j(z_1,...,z_n)$ and power series $b_{ij}$ exist since $\hat\varphi$ takes the ideal $(f_1,...,f_m)$ to $(g_1,...,g_m)$.
Hence, the \'etale approximation of the formal solution as of \cite[Corollary (2.1)]{artin1} provides the result to be proven.
\end{proof}

\begin{proof}[Proof of Theorem~\ref{analytic-result}] 
This goes precisely as the previous proof except that now we approximate the formal solution by an analytic solution, the corresponding approximation theorem by Artin is \cite[Theorem (1.2)]{artin2}.
\end{proof}

\begin{proof}[Proof of Lemma~\ref{rigid-nearby-fibres}] 
Let $\pi:X=\Spec A\ra Y=\Spec B$ denote the map and let $U=Y\setminus\{t=0\}$.
The sheaf $\shT^1_{X/Y}$ is coherent. To prove that it restricts to zero over $f^{-1}(U)$, by Nakayama's lemma, it suffices to show that 
$\shT^1_{X/Y}\otimes_{\shO_X}{\shO_{X,x}/\frak{m}_{X,x}}=0$
for all $x\in X$ with $\pi(x)\in U$.
It suffices to show that if $y\in U$ is a point with $\frak{m}_{Y,y}=(t_1,...,t_r)$ then 
$T^1(A/B,A)\otimes_A A/(t_1,...,t_r)=0$.
For $M$ an $A$-module, we get via \S\ref{change-of-modules} the exact sequence 
$$...\ra T^1(A/B,M/\Ann(t_i))\stackrel{\cdot t_i}\ra T^1(A/B,M)\ra T^1(A/B,M/t_iM)\ra ...$$
from which we learn that 
\begin{equation} \label{inclusion-of-moddy}
T^1(A/B,M)\otimes A/(t_i)\subseteq T^1(A/B,M\otimes A/(t_i)).
\end{equation}
By the assumption and by \S\ref{lemma-basechange}, we have $0=T^1(X_y/y)=T^1(A/B,A/(t_1,...,t_r))$. 
Now by successively using \eqref{inclusion-of-moddy} and Nakayama's Lemma again, we indeed find that $T^1(A/B,A)\otimes_A A/(t_1,...,t_r)=0$ as desired.
\end{proof}


\section{Approximations with non-rigid nearby fibres}
\label{section-divisor}
We next prove variants of the theorems in \S\ref{section-main-results} where $\shT^1_{X/Y}$ is supported in $\{tw=0\}$ where $t\in\Gamma(Y,\shO_Y)$ and $w\in\Gamma(X,\shO_X)$, so nearby fibres need not be rigid. 
To remedy this, in the notation of \S\ref{section-main-results}, we require $I=(w^r)$ for a suitable $r$. 

\begin{setup} \label{the-setup-2}
Let $X\ra Y$ be a flat finite type map of Noetherian affine schemes. 
Let $t\in \Gamma(Y,\shO_Y)=:B$ and $w\in \Gamma(X,\shO_X)=:A$ be non-zero-divisors. 
Let $t$ be a non-zerodivisor in $A/(w^r)$ for all $r$. Set $Z_r:=\Spec A/(w^r)$.
Assume that $\shT^1_{X/Y}$ is supported in $\{tw=0\}$. Let $\widehat X\ra \widehat Y$ be the completion in $t$.
\end{setup}

\begin{theorem}[\bf  formal with a divisor] 
\label{formal-result-divisor}
There exist $N>0$ and $M>0$ so that if $X'\ra Y$ is another flat map of Noetherian affine schemes, $w'\in \Gamma(X',\shO_{X'})$, $Z'_r\subset X'$ given by $(w')^r=0$ and
$\varphi:X'_k\cup Z_r'\ra X_k\cup Z_r$ an isomorphism over $Y$ for some $k>4N$ and $r>4M$ satisfying $\varphi^*(w)=w'$ and hence restricts to an isomorphism $Z_r'\stackrel{\sim}\ra Z_r$ then there is an isomorphism $\widehat\varphi:\widehat X'\ra \widehat X$ that commutes with the maps to $\widehat Y$ and so that the restrictions of $\widehat\varphi$ and $\varphi$ to 
$X'_{k-2N}\cup \hat Z'_{r-2M}$ agree where $\hat Z'_{r-2M}$ is the completion of $Z'_{r-2M}$ in $t$.
\end{theorem}


\begin{proof}
For the largest part, the proof coincides with the proof of Theorem~\ref{formal-result} in \S\ref{proof-section}. 
We only address the adaptions. 
By the support assumption and since $A\cong (w^s)$, there are $N,M\ge 0$ such that 
$$t^Nw^M T^1(A/B,(w^s))=0$$ 
for all $s\ge 0$ and furthermore such that
$\ker_{t^Nw^M} T^2_{A/B}$ is stable in the sense of \eqref{increasing-sequence-to-define-stable}. 
Similar to \S\ref{subsec-prep-proof}, from the short exact sequence
$0\ra A\stackrel{\cdot t^{L}w^{M}}{\lra} (w^{M})\ra A/(t^{L}) \ra 0,$
for $L\ge N$, we then arrive at
$$ t^{2N}w^{2M} T^1(A/B,(w^s)/(t^{L}w^s))=0 $$
for all $s\ge 0$. 
Similarly, as in Lemma~\ref{lemma-shift}, $(w^s)/(t^{L}w^s)$ is isomorphic to  $(t^{l-L}w^s)/(t^{l}w^s)$ as an $A$-module, so
$$ t^{2N}w^{2M} T^1(A/B,(t^{l-L}w^s)/(t^{l}w^s))=0.$$
Since $(t^{l-L}w^s)/(t^{l}w^s)$ has scheme theoretic support in $A/(t^L)$, for $k\ge L$ and using $\varphi^*(w)=w'$, we obtain a form of Lemma~\ref{annnili1} as
\begin{equation} \label{w-version-of-annihil}
 t^{2N}(w')^{2M} T^1(A'/B,(t^{l-L}w^r)/(t^{l}w^r))=0 
\end{equation}
whenever $k\ge L\ge N$ and $l\ge 2L$ and $k,r$ are the integers so that 
$$(\varphi^{-1})^*:A'/(t^k(w')^r)\ra A/(t^kw^r)$$  
is the isomorphism given by assumption. 
The statement of Proposition~\ref{prop-lift-iso} is adapted by setting $I=(w^r), I'=((w')^r)$ and saying that $\bar\beta$ is a map
$A'/(t^{l-2N}(w')^{r-2M})\ra A/(t^{l-2N}w^{r-2M})$.
The proof of the resulting statement is very similar to that of Proposition~\ref{prop-lift-iso} until we arrive at the $B[z_1,...,z_n]/t^{l+1}$-module homomorphism
$$\bar h\in H:= \Hom\left(\frac{J'}{J'^2},\frac{(t^{l-2N}(w')^{r-2M})+J}{(t^{l+1}(w')^{r-2M})+J}\right).$$
This can then be written as $\bar h=t^{2N}(w')^{2M} \bar h'$ which enables the remainder of the proof of Proposition~\ref{prop-lift-iso}
using \eqref{w-version-of-annihil} and then also the proof of Theorem~\ref{formal-result-divisor} just like before.
\end{proof}
\begin{remark} 
The resulting $\hat\varphi$ in Theorem~\ref{formal-result-divisor} does not necessarily map $w$ to $w'$. However, it does map $(w)$ to $(w')$, so we can only conclude that it maps $w$ to $\eps w'$ for $\eps$ a unit satisfying $1=\eps|_{X'_{k-2N}\cup \hat Z'_{r-2M-1}}$.
\end{remark}

For the next theorem, assume we work over a field or excellent discrete valuation ring.
\begin{theorem}[\bf \'etale with a divisor] 
\label{etale-result-divisor}
Let $x\in X$ be a point with $t(x)=w(x)=0$. There exist $N>0,M>0$ with the following property.
If $X'\ra Y$ be another flat map of Noetherian affine schemes, $w'\in \Gamma(X',\shO_{X'})$, $Z'_r\subset X'$ given by $(w')^r=0$ and
$\varphi:X'_k\cup Z_r'\ra X_k\cup Z_r$ an isomorphism over $Y$ for some $k>4N$ and $r>4M$ that restricts to an isomorphism $Z_r'\stackrel{\sim}\ra Z_r$
and has $\varphi^*(w|_{X'_k\cup Z_r'})=w'|_{X_k\cup Z_r}$ then there are \'etale neighbourhoods $U,U'$ of $x$ in $X,X'$ respectively and an isomorphism $\varphi_x:U'\ra U$ that commutes with the maps to $Y$ and so that the restrictions of $\varphi_x$ and $\varphi$ to 
$X'_{k-2N}\cup \hat Z'_{r-2M}$ agree.
\end{theorem}
\begin{proof} 
Given Theorem~\ref{formal-result-divisor}, the proof proceeds precisely as the proof of Theorem~\ref{etale-result} in \S\ref{proof-section} except that we approximate with respect to the ideal $(tw)$ rather than $(t)$.
\end{proof}

Let now $X\ra Y$ be a flat map of complex analytic spaces and $t\in \Gamma(Y,\shO_Y)$ and $w\in \Gamma(X,\shO_X)$ be non-zero-divisors. 
Let $Z,Z_k,X_k\subset X$ be the closed subspaces given by $w=0,w^k=0,t^k=0$ respectively.
Assume that $\shT^1_{X/Y}$ is supported in $tw=0$. 
\begin{theorem}[\bf analytic with a divisor] 
\label{analytic-result-divisor}
Let $x\in X$ be a point with $t(x)=w(x)=0$. There exist $N>0,M>0$ with the following property.
Assume $X'\ra Y$ is another flat map of complex analytic spaces with
$w'\in \Gamma(X',\shO_{X'})$, $Z'_r\subset X'$ given by $(w')^r=0$ and
$\varphi:X'_k\cup Z_r'\ra X_k\cup Z_r$ an isomorphism over $Y$ for some $k>4N$ and $r>4M$ that restricts to an isomorphism $Z_r'\stackrel{\sim}\ra Z_r$
and has $\varphi^*(w|_{X'_k\cup Z_r'})=w'|_{X_k\cup Z_r}$ then there are \'etale neighbourhoods $U,U'$ of $x$ in $X,X'$ respectively and an isomorphism $\varphi_x:U'\ra U$ that commutes with the maps to $Y$ and so that the restrictions of $\varphi_x$ and $\varphi$ to 
$(X'_{k-2N}\cup Z'_{r-2M})\cap U'$ agree.
\end{theorem}

\begin{proof} 
Given Theorem~\ref{formal-result-divisor}, the proof proceeds precisely as the proof of Theorem~\ref{analytic-result} in \S\ref{proof-section} except that we approximate with respect to the ideal $(tw)$ rather than $(t)$.
\end{proof}


\section{Finite determinacy of log morphisms}
\label{section-log-application}
We explain in this section how to deduce from the local uniqueness results in the previous sections a useful fact for fibres of morphisms of log spaces. 

All log structures in this section are considered in the \'etale topology. We will often speak of schemes but all definitions and results equally apply to analytic or algebraic spaces.
For an effective reduced Weil divisor $D\subset X$ in a scheme $X$ let $j:X\setminus D\ra X$ denote the open embedding of its complement.
The \emph{divisorial log structure} $\shM_{(X,D)}$ is defined to be the monoid sheaf
$$\shM_{(X,D)}:= j_*\shO^\times_{X\setminus D}\cap \shO_X$$
together with its natural embeddeding in $\shO_X$. Note that $\shM_{(X,D)}$ contains $\shO^\times_X$. 
A monoid is called \emph{toric} if it is finitely generated, torsion-free, integral and saturated. Recall that a toric monoid has a unique set of minimal generators.
\begin{lemma} \label{constructible}
If $X$ is normal then the sheaf $\overline\shM_{(X,D)}:=\shM_{(X,D)}/\shO_X^\times$ is constructible, i.e. every point has an \'etale neighbourhood that is stratified by locally closed sets on which the sheaf is constant. Furthermore the stalks are toric monoids.
\end{lemma}
\begin{proof} 
We only address the situation where $X$ is an algebraic space, because the same reasoning yields the proof in the situation where $X$ is a complex analytic space.
Set $\overline\shM:=\overline\shM_{(X,D)}$.
Given a geometric point $x\in X$, let $V$ be an (\'etale) neighbourhood of $x$ where every component of $D$ containing $x$ is geometrically unibranch. 
Let $D_1,...,D_r$ be these components. Consider the map $\shM_{(X,D),\bar x}\ra \NN^r$ given by recording the vanishing order of a section along $D_1,...,D_r$. This map factors through $\overline\shM_{\bar x}$. We claim that in fact $\overline\shM_{\bar x}\ra \NN^r$ is injective. Indeed, given two functions $f_1,f_2$ with the same image, the quotient $f_1/f_2$ is invertible on a set whose complement has codimension two, so by the $S_2$ property extends to a section of $\shO^\times_{X,x}$ and this is trivial in $\overline\shM_{\bar x}$.
The group $G$ of Cartier divisors with support in $D_1\cup...\cup D_r$ is a subgroup of $\ZZ^r$ and hence finitely generated and $G\cap \NN^r$ coincides with the image of $\overline\shM_{\bar x}$ in $\NN^r$, so $\overline\shM_{\bar x}$ is integral, saturated and finitely generated. 

We now address the stratification of the chart $V$. 
This will be a refinement of the stratification by the locally closed sets 
$$\shS:=\left\{\left.\left(\bigcap_{i\in I}D_i\right)\setminus \left(\bigcup_{j\in \{1,...,r\}\setminus I}D_j\right)\right| I\subset\{1,...,r\} \right\},$$
indeed let $\eta$ be the generic point of some irreducible component $E$ of a stratum $S_I\in\shS$ for some $I\subset\{1,...,r\}$.
By what we just said before, $\overline\shM|_E$ is a subsheaf of the constant sheaf $\NN^{|I|}$.
The stalk $\overline\shM_{\bar\eta}$ is finitely generated, let $g_1,...,g_n\in \shM_{\bar\eta}$ descend to a set of generators of $\overline\shM_{\bar\eta}$. 
Let $V_i\ra V$ be the open set of definition of $g_i$, so we have an isomorphism
$$\pi:\Gamma(V_1\cap...\cap V_n\cap E,\overline\shM)\ra \overline\shM_{\bar\eta}.$$
Let $\bar y\in V_1\cap...\cap V_n\cap E$ be a point, we want to show that also $\rho:\Gamma(V_1\cap...\cap V_n\cap E,\overline\shM)\ra \overline\shM_{\bar y}$ is an isomorphism. 
We know $\rho$ is injective because $\pi$ factors through $\rho$. This implies surjectivity as well because a section at $x$ with certain vanishing orders along the $D_{i}$ has that same vanishing behaviour at $\bar\eta$.
Hence $\overline\shM$ is constant on $V_1\cap...\cap V_n\cap E$. We obtain a stratification into sets like this by induction on dimension (also applying the above to generic points of components of the locus that wasn't covered in a prior step).
\end{proof}

\begin{remark} 
Note that divisorial log structures are not necessarily coherent in the sense of admitting charts (an assumption often made in the literature). For instance for $X=\Spec \ZZ[x,y,z,w]/(xy-zw)$ and $D=V(z)$, the sheaf $\overline\shM_{(X,D)}$ has rank one at the origin (generated by $z$) but rank two along $x=y=0$ away from the origin, so $\shM_{(X,D)}$ cannot have a chart at the origin, cf. \cite[Ex. 1.11]{logmirror}.
The log structure given here is still \emph{relatively coherent}, \cite[Def.~3.6\,(1)]{rounding}. Theorem~\ref{main-log-result} below covers the yet more general log structures introduced above, e.g. when the normal local model is non-toric.
\end{remark}

\begin{definition}
Let $X$ be normal with a divisorial log structure $\overline\shM_{(X,D)}$, let $\rho_1,...,\rho_s\in \Gamma(X,\overline\shM_{(X,D)})$ be global sections (possibly $s=0$).
Given a geometric point $x\in X$, by the previous Lemma, $\overline\shM_{(X,D),x}$ is a toric monoid that embeds in the monoid of vanishing orders $\NN^r$ where $r$ depends on $x$. 
Let $g^x_1,...,g^x_t$ denote the minimal generators of $\overline\shM_{(X,D),x}$. 
Each of $g^x_1,...,g^x_t,\rho_1,...,\rho_s$ gives rise to an element of $\NN^r$ and we define $m_x\in \NN$ to be the maximum of the components of all these elements.
We then define for any subset $X_0\subset X$,
$$m(X_0,\rho_1,...,\rho_s):=\sup\{m_x|x\in X_0\}.$$
\end{definition}

\noindent Let $f:X\ra Y$ be a finitely presented flat morphism of either
\begin{enumerate}
\item normal complex analytic spaces or 
\item normal algebraic spaces defined over a field or excellent discrete valuation ring with $Y$ locally Noetherian.
\end{enumerate}
In case (1) we additionally assume $f$ to be proper.
Let $D\subset Y$ be a non-empty reduced effective Cartier divisor and $E\subset X$ a possibly empty reduced effective Cartier divisor that is flat over $Y$.
We define log structures $\shM_X:=\shM_{(X,f^{-1}(D)\cup E)}$ and $\shM_Y:=\shM_{(Y,D)}$ on $X$ and $Y$ respectively.
Note that $f$ is compatible with these log structures, i.e. we have a commutative diagram
$$
\vcenter{ \xymatrix{ 
\shO_X & \shM_X \ar[l]\\
f^{-1}\shO_Y\ar[u] & f^{-1}\shM_Y. \ar[u]\ar[l]\\
} }
$$
Let $f^\dagger:(X,\shM_X)\ra (Y,\shM_Y)$ denote the map of log spaces given by the diagram.
Let $0\in D$ be a point, $\rho_1,...,\rho_s\in \overline \shM_{Y,0}$ be the minimal generators and $X_0=f^{-1}(0)$. 
\begin{lemma}
$m(X_0,\rho_1,...,\rho_s)<\infty$
\end{lemma}
\begin{proof} 
By Lemma~\ref{constructible}, $\overline \shM_{X}$ is constructible and, by the assumptions on $f$, $X_0$ permits a finite cover so there are only finitely many stalks of $\overline\shM_{X}$ to consider.
\end{proof}

We denote by $f_0^\dagger:X_0^\dagger \ra 0^\dagger$ the base change of $f^\dagger$ to $0$ as a log morphism. We now recall what this means.
The underlying morphism is just the fibre $f_0:X_0\ra 0$. Consider the commutative diagram
$$
\vcenter{ \xymatrix{ 
\shO_{X_0} & \shM_X|_{X_0} \ar_\alpha[l]\\
f_0^{-1}\shO_{0}\ar[u] & f^{-1}\shM_Y|_0 \ar[u]\ar[l]\\
} }
$$
which constitutes a map of pre-log-spaces.
We transition to the associated log spaces by taking for $\shM_{X_0}$ the pushout $\shM_X|_{X_0}\oplus_{\alpha^{-1}(\shO^\times_{X_0})}\shO^\times_{X_0}$ and similarly for $\shM_0$.
It holds $\overline\shM_X|_{X_0}= \overline\shM_{X_0}$, so this is again constructible.

For $y\in Y$, let $X_y:=f^{-1}(y)$ denote the fibre.
Assume that $E$ is a Cartier divisor and that either
\begin{enumerate}
\item[(A1)] $E_y=X_y\cap E$ and $X_y$ are locally rigid for every point $y\in Y\setminus D$, i.e. $\shT^1_{X_y/y}=0$ 
\item[(A2)] assume that $\shT^1_{X/Y}$ is supported in $f^{-1}(D)\cup E$.
\end{enumerate}
Let $t$ be a local equation for $D$ at $0$.

In case (A1), by the finiteness assumptions and Lemma~\ref{rigid-nearby-fibres}, there exists $N>0$ such that $t^N\shT^1_{X/Y}=0$ holds in a neighbourhood of $X_0$. 
Similarly in case (A2) there exists $N,M>0$, so that locally along $X_0$, for $w$ a local equation of $E$, $t^Nw^M\shT^1_{X/Y}=0$.
By further increasing $N$ if needed, we may assume that $2N\ge m(X_0,\rho_1,...,\rho_2)$ and that $\ker_{t^N}(\shT^2_{X/Y})$ is stable in the sense of \eqref{increasing-sequence-to-define-stable}. 
\begin{theorem} 
\label{main-log-result}
Let $f':X'\ra Y$ be another flat proper map of complex analytic spaces (respectively finitely presented morphism of algebraic spaces) and in case
\begin{enumerate}
\item[(A1)] assume for $k\ge 4N+1$ to have an isomorphism $\varphi: X'_k\cup E'\ra X_k\cup E$ over $D$ where $X_k,X'_k$ denote the base change of $f,f'$ to $(D,\shO_Y/(t^k))$ respectively; and in case
\item[(A2)] assume for $k\ge 4N+1$ and $r\ge 4M+1$ to have an isomorphism $\varphi: X'_k\cup E'_r\ra X_k\cup E_r$ over $D$ with $E_r,E'_r$ the $r$th order neighbourhood of $E,E'$ respectively, so that, locally along $E'$, $\varphi^*(w)=w'$ for $w,w'$ a restriction to domain and target of $\varphi$ of a local equation of $E,E'$ respectively.
\end{enumerate}
The following hold true.
\begin{enumerate}
\item The maps of pairs $(X,f^{-1}(D)\cup E)\ra (Y,0)$ and $(X',(f')^{-1}(D)\cup E')\ra (Y,0)$ are locally along $X_0$ and $X'_0$ isomorphic, so in particular 
\item $X'$ is also normal in a neighbourhood of $X'_0:=(f')^{-1}(0)$ and
\item for $X^\dagger_0$, ${X'_0}^\dagger$ and $0^\dagger$ the log spaces obtained from restricting the divisorial log structures of the pairs $(X,f^{-1}(D)\cup E), (X',(f')^{-1}(D)\cup E'), (Y,D)$ to $X_0,X_0',0$ respectively, there is an isomorphism 
$\varphi^\dagger_0: {X'_0}^\dagger\ra X^\dagger_0$ over $0^\dagger$ whose underlying morphism $\varphi_0$ is the reduction of $\varphi$ modulo $t$.
\end{enumerate}
\end{theorem}

\begin{proof} 
Item (1) directly follows from previous results based on our choice of $N$:
In situation (A1), we conclude (1) from Theorem~\ref{etale-result} in the algebraic space case and from Theorem~\ref{analytic-result} in the analytic case.
In situation (A2), we conclude (1) from Theorem~\ref{etale-result-divisor} in the the algebraic space case and from Theorem~\ref{analytic-result-divisor} in the analytic case.
Part (2) follows from (1). For (3), we first show that $\varphi$ induces a commutative diagram
\[
\xymatrix{
\overline \shM_X|_{X_0} \ar[r] & \varphi_0^{-1}\overline \shM_{X'}|_{X'_0} \\
f^{-1}\overline \shM_Y|_{X_0}\ar[u] \ar[r] & \varphi_0^{-1}{f'}^{-1}\overline \shM_Y|_{X'_0}\ar[u].
}
\]
with horizontal maps being isomorphisms.
By (2) and Lemma~\ref{constructible}, $\overline \shM_X$ and $\overline \shM_Y$ are constructible sheaves. 
By quasi-compactness, we find a finite cover $\{U_\alpha\}_\alpha$ of $X_0$ by (\'etale) open sets of $X$ and similarly $\{U'_\alpha\}_\alpha$ (\'etale) open sets of $X'$ such that $\varphi_\alpha:U'_\alpha\ra U_\alpha$ is an isomorphism over $Y$ restricting to the identity modulo $t^{2N+1}$. Furthermore, we may assume that $\overline \shM_X|_{U_\alpha}$ is generated by global sections and similarly also $\overline \shM_{X'}|_{U'_\alpha}$.

This readily implies that the restrictions to $X_0$ of $\overline \shM_X$ and $\overline \shM_{X'}$ are isomorphic by means of $\varphi$ (compatibly with the map from $f_0^{-1}\overline \shM_{Y,0}$).
Indeed $\varphi_\alpha$ provide isomorphisms locally and to check that these globalize, note that $\varphi_\alpha$ and $\varphi_{\beta}$ differ (additively) by terms of the form $t^{2N+1}g$ for $g$ a holomorphic (respectively regular) function, so by our choice of $N$ in particular the vanishing orders of generators of stalks are preserved but then stalks entirely are preserved since stalks embed in the vanishing order monoids by the proof of Lemma~\ref{constructible}.

In order to produce an isomorphism of the restrictions of $\shM_X$ and $\shM_{X'}$ to $X_0$ compatibly with the maps to $\shO_{X_0}$ and the embeddings of $f^{-1}\shM_Y$ and ${f'}^{-1}\shM_Y$ respectively, we show that the restrictions of the local isomorphisms $\varphi_\alpha$ glue. 
Consider open sets $U_\alpha,U_\beta$ that meet each other, so we have two morphisms
\[
\xymatrix{
\shM_{U_\alpha\cap U_\beta} \ar^{{\varphi_\alpha}\atop{\ }}@<-.5ex>[r] \ar@<.5ex>_{{\ }\atop{\varphi_\beta}}[r] &  \shM_{U'_\alpha\cap U'_\beta}
}
\]
and, because $\shM_{U_\alpha\cap U_\beta}\subset \shO_{U_\alpha\cap U_\beta}$ and $\shM_{U'_\alpha\cap U'_\beta}\subset \shO_{U'_\alpha\cap U'_\beta}$,
compatibly with the structure sheaves, $\varphi_\alpha$ and $\varphi_\beta$ are identical modulo $t^{2N}$ and induce the identity on $\overline\shM_{X_0}$.
Composing yields an automorphism of the source $(\varphi_\beta^{-1}\circ\varphi_\alpha)\in \Aut(\shM_{U_\alpha\cap U_\beta})$ and if we show its restriction to $X_0$ yields the identity then we are done.
Note that both $\varphi_\alpha$ and $\varphi_\beta$ fix $\shO^\times_{X_0}$ pointwise.
Using constructibleness, by \cite[Proposition 2.2]{ScSi06}, the natural inclusion $$\shH om(\overline\shM_{X},\shO^\times_{X})\ra \Aut(\shM_{X})$$ 
that sends $h$ to $(m\mapsto m\cdot h(\overline m))$ is an isomorphism. 
The restriction of an automorphism to $X_0$ is just the image in $\shH om(\overline\shM_{X}|_{X_0},\shO^\times_{X_0})$. 
We see now that $(\varphi_\beta^{-1}\circ\varphi_\alpha)|_{X_0}$ restricts to the identity since the images of the generators in $\shO^\times_{U_\alpha\cap U_\beta}$ are $1$ modulo $t^{2N}$, so in particular $1$ when restricting to $U_\alpha\cap U_\beta\cap X_0$.
\end{proof}

We end the article with an improvement on the uniqueness properties of $\varphi^\dagger_0$ in Theorem~\ref{main-log-result} in case $f$ is relatively log smooth:
recall the notion of a relatively log smooth morphism from \cite[Def.~3.6\,(2)]{rounding}.
\begin{proposition} 
If $f:X^\dagger\ra \Spec(\NN\ra \kk)$ is relatively log smooth for $\kk$ a separably closed field and $\sigma:X\ra X$ an isomorphism of the underlying scheme, then there exists at most one upgrade of $\sigma$ to a log morphism so that $f\circ\sigma=f$ holds as an identity of log morphisms.
\end{proposition} 
\begin{proof}  
We note that \cite[Lemmata (2.1),(2.2)]{kisin} generalize from log smoothness to relative smoothness. 
Indeed, by the definition of relative smoothness, locally there exists a log smooth structure into which the given one embeds. One carefully traces through the arguments in loc.cit. and finds them to work.
\end{proof}



\begin{thebibliography}{77}	

\bibitem[Ar68]{artin2} M.F. Artin: ``On the Solutions of Analytic Equations'', \emph{Inventiones math.} {\bf 5},  1968, 277--291.

\bibitem[Ar69]{artin1} M.F. Artin: ``Algebraic approximation of structures over complete local rings'', \emph{Publications Math\'ematiques de l'IH\'ES} {\bf 36}, 1969, p. 23--58.

\bibitem[AM94]{AM} M.F. Atiyah, I.G. MacDonald: ``Introduction To Commutative Algebra'', Addison-Wesley series in mathematics, Avalon Publishing, 1994, 138p.

\bibitem[FFR]{smoothy} S.~Felten, M.~Filip, H.~Ruddat: \emph{Smoothing toroidal crossing spaces},  \url{https://arxiv.org/pdf/1908.10895}.

\bibitem[GS10]{logmirror} M.~Gross, B.~Siebert: \emph{Mirror symmetry via logarithmic degeneration data II}, J.\ Algebraic Geom. \textbf{19} (2010), 679--780.

\bibitem[Ha10]{hartshorne} R. Hartshorne: ``Deformation theory'', New York Springer, Lecture notes in mathematics {\bf 239}, 1971.

\bibitem[Il94]{illusie} L. Illusie: ``Autour du th\'eor\`eme de monodromie locale'', P\'eriodes p-adiques (Bures-sur-Yvette, 1988), Ast\'erisque \textbf{223}, 1994, 9--57.


\bibitem[KN99]{katonakayama} K. Kato, C. Nakayama: ``Log Betti cohomology, log \'etale cohomology, and log de Rham cohomology of log schemes over C'', Kodai Math. J. \textbf{22}, 1999, 161--186.

\bibitem[Ki03]{kisin} M. Kisin: ``Endomorphisms of logarithmic schemes'', Rend. Sem. Mat. Univ. Padova {\bf 109}, 2003, 247--281.

\bibitem[LS65]{lichtenbaumschlessinger} S. Lichtenbaum, M. Schlessinger: ``The cotangent complex of a morphism'', \emph{Trans. Amer. Math. Soc.} {\bf 128}, 1967, p.41--70. 

\bibitem[MvS02]{MvS02} K. M\"ohring, D. van Straten: ``A Criterion for the Equivalence of Formal Singularities'', American Journal of Mathematics {\bf 124 (6)}, Dec. 2002, p.1319--1327.

\bibitem[Na98]{nakayama} C. Nakayama: ``Nearby cycles for log smooth families'', Compositio Math. \textbf{112}(1), 1998, 45--75.

\bibitem[NO10]{rounding} C. Nakayama; A. Ogus: ``Relative rounding in toric and logarithmic geometry'', Geom. Topol. \textbf{14}(4), 2010, 2189--2241.

\bibitem[RS19]{RS19} H. Ruddat, B. Siebert: ``Period Integrals from Wall Structures via Tropical Cycles, Canonical Coordinates in Mirror Symmetry and Analyticity of Toric Degenerations'', \url{https://arxiv.org/pdf/1907.03794}.

\bibitem[ScSi06]{ScSi06} S. Schr\"oer, B. Siebert: ``Toroidal crossings and logarithmic structures'', \emph{Advances in Mathematics} {\bf 202 (1)}, 1 May 2006, p.189--231.



\end{thebibliography}
\end{document}